\documentclass[12pt]{amsart}

\usepackage{amssymb, amscd, txfonts}
\usepackage{graphicx}

\numberwithin{equation}{section}

\newtheorem{theorem}{Theorem}[section]
\newtheorem{proposition}[theorem]{Proposition}
\newtheorem{lemma}[theorem]{Lemma}

\theoremstyle{definition}

\theoremstyle{remark}
\newtheorem{remark}[theorem]{Remark}

\renewcommand{\hom}{\operatorname{Hom}}
\renewcommand{\ker}{\operatorname{Ker}}

\newcommand{\Z}{\mathbb{Z}}

\newcommand{\C}{\mathbb{C}}

\newcommand{\proj}{{\mathbb P}}

\newcommand{\PGL}{{\rm PGL}}
\newcommand{\SL}{{\rm SL}}
\newcommand{\GL}{{\rm GL}}

\newcommand{\SLSL}{{\rm SL}_2\times{\rm SL}_2}

\newcommand{\Oline}{\mathcal{O}_{{\mathbb P}^{1}}}
\newcommand{\Oplane}{\mathcal{O}_{{\mathbb P}^{2}}}
\newcommand{\Ospace}{\mathcal{O}_{{\mathbb P}^{3}}}
\newcommand{\OPP}{\mathcal{O}_{{\mathbb P}^{1}\times{\mathbb P}^{2}}}
\newcommand{\Orel}{\mathcal{O}_{\pi}}
\newcommand{\OSigma}{\mathcal{O}_{\Sigma}}
\newcommand{\OQ}{\mathcal{O}_{Q}}
\newcommand{\OQQ}{\mathcal{O}_{Q_0}}
\newcommand{\sheaf}{\mathcal{O}}

\DeclareMathOperator{\aut}{Aut}

\begin{document}

\title[]{Rationality of some tetragonal loci}
\author[]{Shouhei Ma}
\address{Graduate~School~of~Mathematics, Nagoya~University, Nagoya 464-8601, Japan}
\email{ma@math.nagoya-u.ac.jp}
\thanks{Supported by Grant-in-Aid for Scientific Research No.12809324 and No.22224001.} 
\subjclass[2000]{Primary 14H45, Secondary 14H10, 14E08}
\keywords{tetragonal curve, moduli space, rationality} 
\maketitle 

\begin{abstract}
We prove that the moduli space of tetragonal curves of genus $g\geq7$ is rational 
when $g \equiv 1, 2, 5, 6, 9, 10$ modulo $12$ and $g\ne9, 45$. 
\end{abstract}

\maketitle


\section{Introduction}\label{sec:intro}

Let $\mathcal{M}_g$ be the moduli space of curves of genus $g\geq7$, 
and $\mathcal{T}_g\subset\mathcal{M}_g$ be the locus of tetragonal curves, 
namely non-hyperelliptic curves which have a map of degree $4$ to ${\proj}^1$. 
Classically $\mathcal{T}_g$ has been known to be unirational (\cite{Pe}, \cite{A-C}, \cite{Sc}), 
but the question whether it is rational had remained open until recently, 
when B\"ohning, Bothmer and Casnati \cite{B-B-C} proved that $\mathcal{T}_7$ is rational. 
In this article we make a further step in this direction, 
showing that $\mathcal{T}_g$ is rational for about half genera. 
 
\begin{theorem}\label{main}
Let $g\geq7$ be a natural number with 
\begin{equation*}\label{eqn: mod 12 condition}
g \equiv 1, 2, 5, 6, 9, 10 \mod 12
\end{equation*}
and $g\ne9, 45$. 
Then the tetragonal locus $\mathcal{T}_g$ is rational. 
\end{theorem} 

This extends the series of rationality results for 
the hyperelliptic loci (\cite{Ka1}, \cite{B-K}) and the trigonal loci (\cite{SB}, \cite{Ma1}, \cite{Ma3}). 
There naturally arises the question at which gonality such a progress should stop. 
One might approach pentagonal loci as well using the description in \cite{Sc}, 
while it seems that only little is known for gonality $\geq6$ (see \cite{Ge}). 
 
One of the basic approaches for proving rationality of a moduli space is 
to first describe it birationally as the quotient of a parameter space $U$ by an algebraic group $G$, 
and then analyze the $G$-action on $U$. 
The first step means to give a construction of general members that is canonical. 
In the present case, we use Schreyer's model (\cite{Sc}) which describes 
a tetragonal curve $C$ as a complete intersection of two relative conics in a ${\proj}^2$-bundle over ${\proj}^1$. 
When $C$ is general, the ambient ${\proj}^2$-bundle $X$ is either 
(i) ${\proj}^1\times{\proj}^2$ or 
(ii) the blow-up of ${\proj}^3$ along a line or 
(iii) a small resolution of a quadric cone in ${\proj}^4$, 
depending on $[g]\in{\Z}/3{\Z}$. 
Thus, in the present case, $U$ is a parameter space of some complete intersection curves in that $X$, and $G$ is the automorphism group of $X$. 
The structure of $U$ varies according to the parity of $g$, 
so the nature of the group action we will study primarily depends on $[g]\in{\Z}/6{\Z}$. 
Moreover, when attacking the rationality problem, 
we were faced with a technical obstruction which caused the further mod $12$ condition in Theorem \ref{main}. 

We work over the complex numbers. 
\S \ref{sec:scroll} contains preliminaries on the relevant ${\proj}^2$-bundles. 
In \S \ref{sec:tetra} we derive a birational description of $\mathcal{T}_g$ as a quotient space. 
\S \ref{sec:technique for rationality} is a collection of miscellaneous techniques for rationality of quotient spaces. 
They will be also useful for other rationality problems. 
Theorem \ref{main} is proved in \S \ref{sec:P1P2} -- \S \ref{sec:quadric cone}: 
this division comes from the above classification (i) -- (iii).

\vspace{0.3cm}

\textbf{Notation.} 
We will use the following notation for irreducible representations of ${\SL}_2$ and ${\SLSL}$: 
\begin{equation*}
V_d = H^0({\Oline}(d)), 
\end{equation*}
\begin{equation*}
V_{d, e} = V_d\boxtimes V_e = H^0({\sheaf}_{{\proj}^1\times{\proj}^1}(d, e)). 
\end{equation*}
The space $V_d$ is also regarded as a ${\GL}_2$-representation in the natural way.


\section{3-dimensional scrolls}\label{sec:scroll}

For two natural numbers $0\leq e\leq f$, 
let $\mathcal{E}_{e,f}$ be the vector bundle 
\begin{equation*}\label{eqn:def rk 3 bundle}
{\Oline} \oplus {\Oline}(-e) \oplus {\Oline}(-f)
\end{equation*}
over ${\proj}^1$, and  
\begin{equation*}\label{eqn:def P^2-bundle}
X_{e,f} = {\proj}\mathcal{E}_{e,f}
\end{equation*}
be the associated ${\proj}^2$-bundle parametrizing lines in the fibers of $\mathcal{E}_{e,f}$. 
We denote by $\pi: X_{e,f}\to{\proj}^1$ the natural projection. 
In the convention of Grothendieck, $X_{e,f}$ is rather the projectivization of 
the dual $\mathcal{E}_{e,f}^{\vee}$. 
Thus $\pi_{\ast}{\Orel}(1)\simeq \mathcal{E}_{e,f}^{\vee}$ 
for the relative hyperplane bundle ${\Orel}(1)$. 
These ${\proj}^2$-bundles play a fundamental role in the study of tetragonal curves. 
In \S \ref{ssec:scroll basic} we recall their basic properties following Schreyer \cite{Sc}. 
When studying birational types of tetragonal loci, we actually use only three ${\proj}^2$-bundles: 
$X_{0,0}={\proj}^1\times{\proj}^2$, $X_{0,1}$ and $X_{1,1}$. 
In \S \ref{ssec:X01} and \S \ref{ssec:X11}, 
we take a closer look at $X_{0,1}$ and $X_{1,1}$. 

\subsection{Basic properties}\label{ssec:scroll basic}

The Picard group of $X_{e,f}$ is freely generated by ${\Orel}(1)$ and $\pi^{\ast}{\Oline}(1)$. 
Accordingly, we will write 
\begin{equation*}
L_{a,b} = {\Orel}(a) \otimes \pi^{\ast}{\Oline}(b). 
\end{equation*}
For example, the canonical bundle of $X_{e,f}$ is isomorphic to $L_{-3,-2+e+f}$: 
this can be seen from the relative Euler sequence 
\begin{equation*}
0 \to {\sheaf}_{X_{e,f}} \to \pi^{\ast}\mathcal{E}_{e,f}\otimes{\Orel}(1) \to T_{\pi} \to 0
\end{equation*}
where $T_{\pi}$ is the relative tangent bundle. 
The intersection numbers between line bundles are calculated from 
\begin{equation}\label{eqn:intersect number}
(L_{1,0}.L_{1,0}.L_{1,0}) = e+f, \quad 
(L_{1,0}.L_{1,0}.L_{0,1}) = 1, \quad 
L_{0,1}.L_{0,1} \equiv 0. 
\end{equation}
%
When $a\geq0$, $b\geq-1$, we have using 
$\pi_{\ast}L_{a,b}\simeq {\rm Sym}^a\mathcal{E}_{e,f}^{\vee}\otimes{\Oline}(b)$ 
that 
\begin{equation*}\label{eqn: h^0(Lab)}
h^0(L_{a,b}) = (e+f)\binom{a+2}{3} + (b+1)\binom{a+2}{2} 
\end{equation*} 
(\cite{Sc}) and $h^i(L_{a,b})=0$ for $i>0$. 

If $b\geq0$ with $(e, f, b)\ne(0, 0, 0)$, the bundle $L_{1,b}$ is base-point-free and the morphism 
\begin{equation*}\label{eqn:scroll embedding}
\phi=\phi_{L_{1,b}} : X_{e,f} \to |L_{1,b}|^{\vee}\simeq{\proj}^N, \quad N=e+f+3b+2, 
\end{equation*}
is birational onto its image. 
It is an embedding if $b>0$. 
The $\pi$-fibers are mapped by $\phi$ isomorphically to planes in $|L_{1,b}|^{\vee}$, which sweep out $\phi(X_{e,f})$. 
The projective variety $\phi(X_{e,f})$ is usually called a \textit{$3$-dimensional rational normal scroll}. 
Its scroll type (\cite{Sc}) is $(b+f, b+e, b)$. 

We will be concerned with the automorphism group of $X_{e,f}$. 
By the relation \eqref{eqn:intersect number}, 
any automorphism acts on ${\rm Pic}(X_{e,f})$ trivially and in particular preserves $\pi$.  
Hence we have the basic exact sequence 
\begin{equation}\label{eqn:basic sequence}
1 \to {\aut}(\mathcal{E}_{e,f})/{\C}^{\times} \to {\aut}(X_{e,f}) \to {\PGL}_2 \to 1
\end{equation}
where ${\aut}(\mathcal{E}_{e,f})$ is the group of bundle automorphisms 
which are the identity over the base. 
In this article we refrain from working with ${\aut}(X_{e,f})$ for general $(e, f)$ 
and restrict ourselves to $X_{0,0}$, $X_{0,1}$ and $X_{1,1}$, giving an ad hoc treatment. 
Note that ${\aut}(X_{0,0})$ is just ${\PGL}_2\times{\PGL}_3$. 
The other two cases are studied in \S \ref{ssec:X01} and \S \ref{ssec:X11}. 
Here we just mention the following general duality.

\begin{lemma}\label{duality}
We have an isomorphism ${\aut}(X_{e,f})\simeq{\aut}(X_{f-e,f})$ 
of algebraic groups. 
\end{lemma}

\begin{proof}
It is convenient to consider the double cover 
\begin{equation*}
\tilde{G} = {\SL}_2\ltimes({\aut}(\mathcal{E}_{e,f})/{\C}^{\times})
\end{equation*}
of  ${\aut}(X_{e,f})$, 
where ${\SL}_2$ acts on $X_{e,f}$ and ${\aut}(\mathcal{E}_{e,f})$ through the ${\SL}_2$-linearization of $\mathcal{E}_{e,f}$.  
The kernel of the natural covering map $\tilde{G}\to{\aut}(X_{e,f})$ is generated by $(-1, (-1)^{\ast})$. 
On the other hand, by the canonical isomorphism ${\aut}(\mathcal{E}_{e,f})\simeq{\aut}(\mathcal{E}_{e,f}^{\vee})$ 
and by the ${\SL}_2$-linearization of $\mathcal{E}_{e,f}^{\vee}$, 
we have a surjective homomorphism $\tilde{G}\to{\aut}({\proj}\mathcal{E}_{e,f}^{\vee})$. 
Its kernel is also generated by $(-1, (-1)^{\ast})$. 
Hence we have ${\aut}(X_{e,f})\simeq{\aut}({\proj}\mathcal{E}_{e,f}^{\vee})$. 
Finally, ${\proj}\mathcal{E}_{e,f}^{\vee}$ is canonically isomorphic to $X_{f-e,f}$. 
\end{proof}

\subsection{$X_{0,1}$ as a blown-up ${\proj}^3$}\label{ssec:X01}

The ${\proj}^2$-bundle $X_{0,1}={\proj}\mathcal{E}_{0,1}$, 
$\mathcal{E}_{0,1}={\sheaf}_{{\proj}^1}^{\oplus2}\oplus{\Oline}(-1)$, 
has the special surface $\Sigma = {\proj}{\sheaf}_{{\proj}^1}^{\oplus2}$ 
which is invariant under ${\aut}(X_{0,1})$. 
Since a section of ${\Oline}(1)\subset\mathcal{E}_{0,1}^{\vee}$ defines the divisor $\Sigma+F\in|{\Orel}(1)|$ where $F$ is a $\pi$-fiber, 
$\Sigma$ is (the unique) member of $|L_{1,-1}|$. 
We shall distinguish the two rulings on $\Sigma\simeq{\proj}^1\times{\proj}^1$ by 
letting $\pi|_{\Sigma}\colon\Sigma\to{\proj}^1$ be the first projection, 
and the other be the second. 
In particular, $L_{0,1}|_{\Sigma}\simeq{\OSigma}(1, 0)$. 
By the adjunction formula we see that $L_{1,0}|_{\Sigma}\simeq{\OSigma}(0, 1)$.

\begin{lemma}
The morphism 
\begin{equation*}
\phi = \phi_{{\Orel}(1)}\colon X_{0,1} \to |{\Orel}(1)|^{\vee}\simeq{\proj}^3
\end{equation*} 
is the blow-up along a line $l\subset{\proj}^3$ with exceptional divisor $\Sigma$, 
and $\pi\colon X_{0,1}\to{\proj}^1$ is obtained as the resolution of the projection 
${\proj}^3\dashrightarrow{\proj}^1$ from $l$.  
\end{lemma}

\begin{proof}
We see that $\phi$ is birational because $({\Orel}(1))^3=1$. 
Since $|{\Orel}(1)| \vert_{\Sigma}=|{\OSigma}(0, 1)|$, 
$\phi$ maps $\Sigma$ to a line $l$, contracting the second ruling 
and mapping the first ruling fibers isomorphically to $l$. 
On the other hand, each $\pi$-fiber is mapped isomorphically to a plane containing $l$. 
This implies our claim. 
\end{proof}

Since $L_{a,b}\simeq\phi^{\ast}{\Ospace}(a+b)\otimes{\sheaf}_{X_{0,1}}(-b\Sigma)$, 
we can identify $|L_{a,b}|$ with the linear system of 
surfaces of degree $a+b$ in ${\proj}^3$ which have multiplicity $\geq b$ along $l$. 
To describe it explicitly, 
take homogeneous coordinates $[X_0,\cdots, X_3]$ of ${\proj}^3$ 
and let $l$ be defined by $X_0=X_1=0$. 
Then the subspace $H^0(L_{a,b})\subset H^0({\Ospace}(a+b))$ is given by 
\begin{equation}\label{eqn:filtration H0(X01Lab)}
\mathop{\bigoplus}_{i=b}^{a+b} V_i(X_0, X_1)\otimes V_{a+b-i}(X_2, X_3), 
\end{equation}
where $V_d(X_s, X_t)$ denotes the space of homogeneous polynomials of degree $d$ in variables $X_s, X_t$.  
  
We can regard ${\aut}(X_{0,1})$ as the subgroup of ${\PGL}_4$ stabilizing $l$. 
It is convenient to consider inside ${\GL}_4$ the following double cover of ${\aut}(X_{0,1})$: 
\begin{equation*}\label{eqn:cover grp X01 I}
\tilde{G} = 
\left\{ \;  \begin{pmatrix} g_1 & 0 \\
                                        h & g_2 \end{pmatrix} \in {\GL}_4   \Biggm\vert  
g_1\in{\SL}_2, \:  g_2\in{\GL}_2, \:  h\in {\rm M}_{2,2} \right\}. 
\end{equation*}
This group is naturally isomorphic to the semidirect product 
\begin{equation*}\label{eqn:cover grp X01 II}
\tilde{G} \simeq ({\SL}_2\times{\GL}_2) \ltimes {\hom}(V_1^{(1)}, V_1^{(2)}), 
\end{equation*}
where 
$V_1^{(1)}={\C}\langle X_0, X_1\rangle$ and $V_1^{(2)}={\C}\langle X_2, X_3\rangle$ are two copies of $V_1$, 
${\GL}_2$ acts on $V_1^{(2)}$ in the standard way, 
and ${\SL}_2$ acts on $(V_1^{(1)})^{\vee}$ by the dual representation of its standard action on $V_1^{(1)}$. 
The kernel of the projection $\tilde{G}\to{\aut}(X_{0,1})$ is generated by $(-1, -1)\in{\SL}_2\times{\GL}_2$. 

Now $H^0(L_{a,b})$ is a $\tilde{G}$-representation, 
and \eqref{eqn:filtration H0(X01Lab)} gives the irreducible decomposition 
under the subgroup ${\SL}_2\times{\GL}_2\subset\tilde{G}$: 
the $i$-th summand in \eqref{eqn:filtration H0(X01Lab)} is the  ${\SL}_2\times{\GL}_2$-representation $V_{i,a+b-i}$. 
To express the action of the unipotent radical ${\hom}(V_1, V_1)$, for $h\in {\hom}(V_1, V_1)$ we set 
\begin{equation*}
{\rm exp}(h)=(1, h, h^{\otimes2}/2,\cdots)\in \mathop{\oplus}_{d\geq0} {\hom}(V_1, V_1)^{\otimes d}. 
\end{equation*}
Then $h$ acts on $H^0(L_{a,b})$ by the linear maps 
\begin{equation}\label{eqn:unipotent radical action X01 general}
\langle \cdot,\, {\rm exp}(h) \rangle : V_{i,a+b-i} \to \mathop{\oplus}_{d\geq0} V_{i+d,a+b-i-d}  
\end{equation}
induced from the multiplication $V_i\times V_1\to V_{i+1}$ and the contraction $V_j\times V_1^{\vee}\to V_{j-1}$. 
In particular, the subspace $F_i=\oplus_{j\geq i}V_{j,a+b-j}$ is $\tilde{G}$-invariant. 
It is the space of polynomials of degree $a+b$ vanishing of order $\geq i$ along $l$: 
that is, $F_i\simeq H^0(L_{a+b-i,i})$. 
Geometrically the quotient map $H^0(L_{a,b})\to H^0(L_{a,b})/F_i$ gives 
the $\leq(i-b-1)$-th Taylor development along $\Sigma$ of the sections of $L_{a,b}$. 
Here note that $L_{a,b}|_{\Sigma}\otimes(N_{\Sigma/X_{0,1}})^{-k}\simeq \mathcal{O}_{\Sigma}(b+k, a-k)$. 

We remark that $L_{a,b}$ admits a $\tilde{G}$-linearization 
through that of ${\Ospace}(a+b)$ and the ideal sheaf $\mathcal{I}_l^b$ (for $b\geq0$). 
Since the element $(-1, -1)\in{\SL}_2\times{\GL}_2$ acts on $L_{a,b}$ by multiplication by $(-1)^{a+b}$, we see that 

\begin{lemma}\label{linearization X01}
The bundle $L_{a,b}$ is ${\aut}(X_{0,1})$-linearized when $a+b$ is even. 
\end{lemma}



\subsection{$X_{1,1}$ as a small resolution of a quadric cone}\label{ssec:X11}

The ${\proj}^2$-bundle $X_{1,1}={\proj}\mathcal{E}_{1,1}$, 
$\mathcal{E}_{1,1}={\Oline}\oplus{\Oline}(-1)^{\oplus2}$, 
has the special section $\sigma={\proj}{\Oline}$ which is invariant under ${\aut}(X_{1,1})$. 
Since $({\Orel}(1).\sigma)=0$, the morphism 
\begin{equation*}\label{eqn:bl-up cone}
\phi=\phi_{{\Orel}(1)} : X_{1,1} \to |{\Orel}(1)|^{\vee}\simeq{\proj}^4 
\end{equation*}
contracts $\sigma$ to a point, say $p_0$.  

\begin{lemma}
The image $Q=\phi(X_{1,1})$ is the quadric cone 
over a smooth quadric $Q_0\subset{\proj}^3$ with vertex $p_0$, 
and $\phi:X_{1,1}\to Q$ is a small resolution of $p_0$ with exceptional curve $\sigma$. 
\end{lemma} 

\begin{proof}
Since $({\Orel}(1))^3=2$ and $Q$ is nondegenerate, 
$Q$ must be a quadric hypersurface and $\phi:X_{1,1}\to Q$ is birational. 
The $\pi$-fibers are mapped isomorphically to planes, 
which intersect with each other at $p_0$. 
Swept out by those planes, $Q$ must be a quadric cone with vertex $p_0$. 
\end{proof}

Let $f\colon Q\dashrightarrow Q_0$ be the projection from $p_0$. 
Via the pullback by $f$, the two rulings on $Q_0\simeq{\proj}^1\times{\proj}^1$ correspond to 
the two families of planes on $Q$ which pass through $p_0$. 
We shall distinguish them so that $\pi:X_{1,1}\to{\proj}^1$ is the resolution of 
$Q\dashrightarrow Q_0\stackrel{\pi_1}{\to}{\proj}^1$ where $\pi_1$ is the ``first'' projection. 
In other words, the ``first'' family is the $\phi$-image of $|L_{0,1}|$.  
On the other hand, the ``second'' family gives rise to $|L_{1,-1}|$, 
whose member is the blow-up of such a plane at $p_0$ and contains $\sigma$ as the $(-1)$-curve. 
The composition 
\begin{equation*}
X_{1,1}\stackrel{\phi}{\to}Q\stackrel{f}{\dashrightarrow}Q_0\simeq{\proj}^1\times{\proj}^1 
\end{equation*}
is given by the relative projection from $\sigma$ of the ${\proj}^2$-bundle $X_{1,1}/{\proj}^1$.


In order to describe ${\aut}(X_{1,1})$, consider the blow-up $\hat{Q}\to X_{1,1}$ along $\sigma$. 
$\hat{Q}$ is the blow-up of $Q$ at $p_0$ and so is 
the ${\proj}^1$-bundle ${\proj}({\OQQ}(1)\oplus{\OQQ})$ over $Q_0$ with exceptional divisor ${\proj}{\OQQ}(1)$. 
As in \eqref{eqn:basic sequence}, we have the exact sequence 
\begin{equation*}
1 \to {\aut}({\OQQ}(1)\oplus{\OQQ})/{\C}^{\times} \to {\aut}(\hat{Q}) \to {\aut}(Q_0) \to 1. 
\end{equation*}
We may identify the quotient group ${\aut}({\OQQ}(1)\oplus{\OQQ})/{\C}^{\times}$ with 
the subgroup $R\subset{\aut}({\OQQ}(1)\oplus{\OQQ})$  
consisting of isomorphisms of the form 
$\begin{pmatrix}\alpha & s \\ 0 & 1\end{pmatrix}$ 
where $\alpha\in{\C}^{\times}$ and $s\in{\hom}({\OQQ}, {\OQQ}(1))$. 
In particular, $R\simeq {\C}^{\times}\ltimes H^0({\OQQ}(1))$. 
Now ${\aut}(X_{1,1})$ is the identity component of ${\aut}(Q)={\aut}(\hat{Q})$. 
We have its natural ${\Z}/2\times{\Z}/2$-covering 
\begin{equation*}
\tilde{G}' = ({\SL}_2\times{\SL}_2) \ltimes R \simeq ({\SL}_2\times{\SL}_2\times{\C}^{\times})\ltimes V_{1,1}. 
\end{equation*}
Dividing $\tilde{G}'$ by $(1, -1, -1)\in({\SL}_2)^2\times{\C}^{\times}$, 
we obtain a double cover $\tilde{G}$ of ${\aut}(X_{1,1})$ isomorphic to $({\SL}_2\times{\GL}_2)\ltimes V_{1,1}$. 
The kernel of the projection $\tilde{G}\to{\aut}(X_{1,1})$ is generated by $(-1, -1)\in{\SL}_2\times{\GL}_2$. 

Every line bundle on $X_{1,1}$ is obtained as the extension of that on $X_{1,1}\backslash \sigma=Q\backslash p_0$, 
which in turn is the pullback by $f$ of that on $Q_0$. 
Explicitly, we have 
\begin{equation}\label{eqn:LB X11 cone}
L_{a,b} \simeq f^{\ast}{\OQQ}(b, 0)\otimes {\sheaf}_Q(a) \simeq f^{\ast}{\OQQ}(a+b, a) 
\end{equation}
over $Q\backslash p_0$. 

As in \S \ref{ssec:X11}, 
$H^0(L_{a,b})$ is a $\tilde{G}'$-representation and has the invariant filtration 
\begin{equation*}\label{eqn:filtration X11 gene}
0\subset F_a\subset F_{a-1}\subset \cdots \subset F_1\subset H^0(L_{a,b}) 
\end{equation*}
where $F_i$ is the space of sections vanishing of order $\geq i$ along $\sigma$. 
To be more explicit, 
we take bi-homogeneous coordinates $([X_0, X_1], [Y_0, Y_1])$ of $Q_0$ 
and homogeneous coordinates $[Z, Z_{00}, Z_{01}, Z_{10}, Z_{11}]$ of ${\proj}^4$ where  
$Z_{ij}=X_iY_j$ and $p_0=[1, 0,\cdots, 0]$. 
By \eqref{eqn:LB X11 cone} we may identify $H^0(L_{a,b})$ with 
\begin{equation}\label{eqn:irr decomp Lab X11}
\mathop{\bigoplus}_{i=0}^{a} V_{b+i}(X_0, X_1)\otimes V_{i}(Y_0, Y_1)Z^{a-i}. 
\end{equation}
This expression is the irreducible decomposition under the subgroup $({\SL}_2)^2\subset\tilde{G}'$, 
the $i$-th summand isomorphic to $V_{b+i,i}$. 
We then have $F_i=\oplus_{j\geq i}V_{b+j,j}$. 
The torus ${\C}^{\times}$ acts on $V_{b+i,i}$ by weight $i-a$. 
(Tensoring \eqref{eqn:irr decomp Lab X11} with the weight $a$ scalar representation of ${\C}^{\times}$, 
we obtain a $\tilde{G}$-representation on $H^0(L_{a,b})$.)  
The unipotent radical $V_{1,1}\ni h$ acts by the multiplication maps 
\begin{equation*}
\cdot{\exp}(h) : V_{b+i,i} \to \mathop{\oplus}_{d=0}^{a-i} V_{b+i+d,i+d}  
\end{equation*}
where ${\exp}(h)=(1, h, h^{\otimes2}/2, \cdots)\in\oplus_{d\geq0}V_{d,d}$. 
In particular, the quotient representation $H^0(L_{a,b})/F_i$ is isomorphic to $H^0(L_{i-1,b})$. 
Geometrically the quotient map $H^0(L_{a,b})\to H^0(L_{a,b})/F_i$ gives 
the $\leq(i-1)$-th Taylor development of the sections of $L_{a,b}$ along the exceptional divisor $E$ of $\hat{Q}$. 
Here notice that $L_{a,b}|_{E}\otimes(N_{E/\hat{Q}})^{-k}\simeq \mathcal{O}_{E}(b+k, k)$.

We remark that 

\begin{lemma}\label{linearization X11}
The line bundle $L_{a,b}$ is ${\aut}(X_{1,1})$-linearized when $b$ is even. 
\end{lemma}

\begin{proof}
We have natural ${\aut}(X_{1,1})$-linearizations on 
$L_{0,-2}=\pi^{\ast}K_{{\proj}^1}$, $L_{-2,0}=f^{\ast}K_{Q_0}$ and $L_{-3,0}=K_{X_{1,1}}$. 
When $b$ is even, $L_{a,b}$ can be written as a tensor product of these bundles.  
\end{proof}


\begin{remark}
More generally, on $X_{e,f}$ with $e\ne0$ (resp. $0=e<f$), 
the bundle $L_{a,b}$ is ${\aut}(X_{e,f})$-linearized if $b$ (resp. $af+b$) is even. 
\end{remark}


\section{Tetragonal loci}\label{sec:tetra}

In this section we follow Schreyer's description \cite{Sc} of tetragonal curves to derive a birational model of 
the tetragonal locus $\mathcal{T}_g$ as a quotient space. 
Notice that we are assuming $g\geq7$.  
First recall some basic facts: 
\begin{itemize}
\item A tetragonal curve $C$ is not trigonal; 
\item When $C$ is general in $\mathcal{T}_g$, its tetragonal pencil $C\to{\proj}^1$ is unique; 
\item If $g\geq10$, $C$ has a unique tetragonal pencil precisely when $C$ is not bielliptic; 
\item $\mathcal{T}_g$ is irreducible of dimension $2g+3$. 
\end{itemize}
The first three properties can be seen by looking at the product $C\to{\proj}^1\times{\proj}^1$ of two pencils. 
See \cite{A-C} and the references therein for the last property. 

Now let $\pi:C\to{\proj}^1$ be a tetragonal map. 
We regard $C$ as canonically embedded in ${\proj}^{g-1}$. 
For each $p\in{\proj}^1$, the (possibly infinitely near) four points $\pi^{-1}(p)$ span a plane in ${\proj}^{g-1}$ by Riemann-Roch. 
The $3$-fold swept out by those planes is a rational normal scroll: we may write it as $\phi_{L_{1,n}}(X_{e,f})$ 
for some $0\leq e\leq f$ and $n\geq0$. 
If we view $C$ as a curve on $X_{e,f}$, 
it turns out to be a complete intersection of two surfaces in $|L_{2,b}|$, $|L_{2,c}|$ for some $b\leq c$ (\cite{Sc}). 
Comparing the adjunction formula for $C$ with the relation $L_{1,n}|_C\simeq K_C$, we see that 
\begin{equation*}\label{eqn:scroll degree from C}
n = b+c+e+f-2. 
\end{equation*}
Calculating ${\deg}K_C=(L_{1,n}.L_{2,b}.L_{2,c})$, we obtain 
\begin{equation}\label{eqn:relation of invariants}
g = 4(e+f) + 3(b+c-1), 
\end{equation}
which imposes a relation between $(e, f)$ and $(b, c)$. 

Let $\hat{\mathcal{T}}_g$ be the moduli space of tetragonal curves of genus $g\geq7$ \textit{given} with a tetragonal pencil $C\to{\proj}^1$. 
The natural projection $\hat{\mathcal{T}}_g\to\mathcal{T}_g$ is birational. 
For $0\leq e\leq f$ and $b\leq c$, let 
$\hat{\mathcal{T}}_g(e, f; b, c)\subset\hat{\mathcal{T}}_g$ be the locus of those $C\to{\proj}^1$ which lies on $X_{e,f}$ 
as a complete intersection of surfaces in $|L_{2,b}|$ and $|L_{2,c}|$. 
Then we have the stratification 
\begin{equation}\label{eqn:stratification}
\hat{\mathcal{T}}_g = \mathop{\sqcup}_{(e,f;b,c)}\hat{\mathcal{T}}_g(e, f; b, c)
\end{equation}
where $(e, f)$ and $(b, c)$ satisfy \eqref{eqn:relation of invariants}. 
This presentation still includes many empty strata (see \cite{Sc}), but we do not mind this redundancy here. 

Since the embedding $C\subset X_{e,f}$ is canonical, we find that each stratum is 
the quotient by ${\aut}(X_{e,f})$ of the parameter space of those complete intersection curves. 
More precisely, when $b=c$, we have 
\begin{equation*}\label{eqn:stratum as Grassmann quotient}
\hat{\mathcal{T}}_g(e,f; b, b) \sim \mathbb{G}(1, |L_{2,b}|)/{\aut}(X_{e,f}) 
\end{equation*}
where $\mathbb{G}(1, |L_{2,b}|)$ is the Grassmannian of pencils in $|L_{2,b}|$. 
On the other hand, when $b<c$, the surface $S\in|L_{2,b}|$ containing $C$ is unique, 
while those in $|L_{2,c}|$ are unique up to $S+|L_{0,c-b}|$. 
To express this situation, let $\mathcal{L}\to|L_{2,b}|$ be the tautological bundle 
and $\mathcal{E}\to|L_{2,b}|$ be the quotient bundle 
\begin{equation}\label{eqn:VB describing c.i.}
\mathcal{E} = \underline{H^0(L_{2,c})}/\mathcal{L}\otimes H^0(L_{0,c-b}),
\end{equation}
where $\underline{H^0(L_{2,c})}$ means the product bundle $H^0(L_{2,c})\times|L_{2,b}|$, 
and the bundle homomorphism $\mathcal{L}\otimes H^0(L_{0,c-b})\to\underline{H^0(L_{2,c})}$ is induced by 
the multiplication map $H^0(L_{2,b})\times H^0(L_{0,c-b})\to H^0(L_{2,c})$. 
Then we have 
\begin{equation*}\label{eqn:stratum as bundle quotient}
\hat{\mathcal{T}}_g(e, f; b, c) \sim {\proj}\mathcal{E}/{\aut}(X_{e,f}). 
\end{equation*}

In order to study the birational type of $\mathcal{T}_g$, we want to identify the largest stratum in \eqref{eqn:stratification}. 
This was done by del Centina and Gimigliano \cite{dC-G}. 
The result depends on the congruence of $g$ modulo $6$ and is summarized as follows. 
(See also \cite{C-dC} \S 3.) 

\begin{proposition}\label{tetra loci}
Let $\mathcal{E}$ be the bundle defined in \eqref{eqn:VB describing c.i.} with $c=b+1$. 

\noindent
$(0)$ When $g\equiv0 \; (6)$, denote $g=6b$. 
Then $(L_{2,b}, L_{2,b+1})$ complete intersections in $X_{0,0}={\proj}^1\times{\proj}^2$ give the largest stratum. 
Hence $\mathcal{T}_{6b}\sim{\proj}\mathcal{E}/{\PGL}_2\times{\PGL}_3$. 

\noindent
$(1)$ When $g\equiv1 \: (6)$, denote $g=6b+1$. 
Then $(L_{2,b}, L_{2,b})$ complete intersections in $X_{0,1}$ give the largest stratum. 
Thus $\mathcal{T}_{6b+1}\sim\mathbb{G}(1, |L_{2,b}|)/{\aut}(X_{0,1})$. 

\noindent
$(2)$ When $g\equiv2 \; (6)$, denote $g=6b+8$. 
Then $(L_{2,b}, L_{2,b+1})$ complete intersections in $X_{1,1}$ give the largest stratum. 
Hence $\mathcal{T}_{6b+8}\sim{\proj}\mathcal{E}/{\aut}(X_{1,1})$. 

\noindent
$(3)$ When $g\equiv3 \; (6)$, denote $g=6b-3$. 
Then $(L_{2,b}, L_{2,b})$ complete intersections in $X_{0,0}$ give the largest stratum. 
Thus $\mathcal{T}_{6b-3}\sim\mathbb{G}(1, |L_{2,b}|)/{\PGL}_2\times{\PGL}_3$. 

\noindent
$(4)$ When $g\equiv4 \: (6)$, denote $g=6b+4$. 
Then $(L_{2,b}, L_{2,b+1})$ complete intersections in $X_{0,1}$ give the largest stratum. 
Hence $\mathcal{T}_{6b+4}\sim{\proj}\mathcal{E}/{\aut}(X_{0,1})$. 

\noindent
$(5)$ When $g\equiv5 \; (6)$, denote $g=6b+5$. 
Then $(L_{2,b}, L_{2,b})$ complete intersections in $X_{1,1}$ give the largest stratum. 
Thus $\mathcal{T}_{6b+5}\sim\mathbb{G}(1, |L_{2,b}|)/{\aut}(X_{1,1})$. 
\end{proposition}

\begin{proof}
Here let us give a self-contained argument. 
It is sufficient to check that the above quotients have dimension $2g+3$, 
and this follows from the formulae
\begin{equation*}
{\dim}{\aut}(X_{0,0}) = {\dim}{\aut}(X_{0,1}) = {\dim}{\aut}(X_{1,1}) = 11, 
\end{equation*}
\begin{equation*}
h^0(L_{2,b}) = 4(e+f) + 6(b+1), 
\end{equation*}
from \S \ref{sec:scroll}.  
\end{proof}

In \S \ref{sec:P1P2}--\S \ref{sec:quadric cone}, 
we use these descriptions of $\mathcal{T}_g$ to prove Theorem \ref{main}. 
In order to have ${\aut}(X_{e,f})$-linearizations on some vector bundles, 
we were forced to assume $b$ to be even in cases (1), (3), (5), and odd in cases (0), (2), (4). 
This caused the further mod 12 classification in Theorem \ref{main}.


\section{Supplementary techniques for rationality}\label{sec:technique for rationality}

In this section we collect some techniques for proving rationality of quotient varieties 
that supplement the basic ones as in \cite{Bo} and that will be used repeatedly in the rest of the article. 
We encourage the reader to skip for the moment and return when necessary. 
Most of this section is more or less standard, but for the convenience of the reader we sketched some proof.

\subsection{Quotients of Grassmannians}\label{ssec:Grass quot}

Let $G$ be an algebraic group and $V$ be a $G$-representation. 
We denote by $G_0\subset G$ the subgroup of elements which act on $V$ by scalar multiplication, 
and set $\bar{G}=G/G_0$. 
Let $G(a, V)=\mathbb{G}(a-1, {\proj}V)$ be the Grassmannian of $a$-dimensional linear subspaces in $V$. 
As shown in Proposition \ref{tetra loci}, we will be interested in the problem 
whether the quotient $G(a, V)/G$ is rational, or at least stably rational of small level. 
First notice that we have a natural birational identification 
\begin{equation}\label{eqn:Grass linear}
G(a, V)/G \sim {\hom}({\C}^a, V)/{\GL}_a\times G = (({\C}^a)^{\vee}\boxtimes V)/{\GL}_a\times G, 
\end{equation}
so that the problem could be reduced to the case of linear quotients. 

To prove stable rationality of $G(a, V)/G$, 
it is useful to consider the universal subbundle $\mathcal{E}\to G(a, V)$ of rank $a$, 
which is $G$-linearized. 
Its projectivization is viewed as the correspondence 
\begin{equation*}
{\proj}\mathcal{E} = \{ (P, [v])\in G(a, V)\times{\proj}V \; |\; {\C}v\subset P \}
\end{equation*}
between $G(a, V)$ and ${\proj}V$. 
The second projection ${\proj}\mathcal{E}\to{\proj}V$ is identified with 
the relative Grassmannian $G(a-1, \mathcal{F})$ over ${\proj}V$, 
where $\mathcal{F}\to{\proj}V$ is the universal quotient bundle of rank ${\dim}V-1$. 

\begin{lemma}\label{Grass correspo}
Suppose that $\bar{G}$ acts on $G(a, V)$ almost freely and that 
$G_0$ acts on $\mathcal{E}\otimes({\det}\mathcal{E})^d$ trivially for some $d\in{\Z}$. 

\noindent
(1) If furthermore $\bar{G}$ acts on ${\proj}V$ almost freely, then 
\begin{equation*}
{\proj}^{a-1}\times(G(a, V)/G) \sim G(a-1, {\dim}V-1)\times({\proj}V/G). 
\end{equation*}
(2) If $G$ acts on ${\proj}V$ almost transitively with $H\subset G$ the stabilizer of a general point $[v]\in{\proj}V$, then 
\begin{equation*}
{\proj}^{a-1}\times(G(a, V)/G) \sim G(a-1, V/{\C}v)/H. 
\end{equation*}
\end{lemma}

\begin{proof}
Note that we have a canonical identification ${\proj}\mathcal{E}={\proj}(\mathcal{E}\otimes({\det}\mathcal{E})^d)$. 
Using the no-name lemma for the bundle $\mathcal{E}\otimes({\det}\mathcal{E})^d$ which is $\bar{G}$-linearized, 
we obtain ${\proj}\mathcal{E}/G\sim{\proj}^{a-1}\times(G(a, V)/G)$. 
On the other hand, the bundle $\mathcal{F}\otimes{\sheaf}_{{\proj}V}(1)$ over ${\proj}V$ is always $\bar{G}$-linearized, 
and we have $G(a-1, \mathcal{F})=G(a-1, \mathcal{F}\otimes{\sheaf}_{{\proj}V}(1))$. 
Now (1) is a consequence of the no-name lemma applied to $\mathcal{F}\otimes{\sheaf}_{{\proj}V}(1)$, 
while (2) follows from the slice method for the projection $G(a-1, \mathcal{F})\to{\proj}V$. 
\end{proof}

Next consider the situation where we have a surjective $G$-homomorphism $f:V\to W$ to another $G$-representation $W$. 
Notice that we are not assuming $V$ to be completely reducible. 
We have a natural dominant map 
\begin{equation}\label{eqn:Grass map}
G(a, V)\dashrightarrow G(a, W), 
\end{equation}
whose fiber over $P\in G(a, W)$ is an open set of $G(a, f^{-1}(P))$. 
Let $\mathcal{G}\to G(a, W)$ be the universal subbundle for $G(a, W)$, 
and $\mathcal{H}\to G(a, W)$ be the vector bundle obtained as the inverse image of $\mathcal{G}$ 
by the bundle homomorphism $\underline{f}:\underline{V}\to\underline{W}$ over $G(a, W)$. 
Then \eqref{eqn:Grass map} induces a $G$-equivariant birational map 
\begin{equation*}
G(a, V)\dashrightarrow G(a, \mathcal{H})
\end{equation*}
to the relative Grassmannian $G(a, \mathcal{H})$. 
As in the proof of Lemma \ref{Grass correspo}, 
we obtain by the no-name lemma the following. 

\begin{lemma}\label{Grass no-name}
Suppose that $\bar{G}$ acts on $G(a, W)$ almost freely and 
that $G_0$ acts on $\mathcal{H}\otimes({\det}\mathcal{G})^d$ trivially for some $d\in{\Z}$. 
Then, setting $n_0={\dim}V-{\dim}W$, we have  
\begin{equation*}
G(a, V)/G \sim G(a, n_0+a)\times(G(a, W)/G). 
\end{equation*}
\end{lemma}
 
This can be seen as a Grassmannian version of the no-name method. 
  
In the above lemmas, we are required to check almost-freeness of an action on a Grassmannian. 
In many cases it follows from the following observation. 
(This can also be found essentially in the proof of Proposition 1.3.2.10 in \cite{Bo}.)  

\begin{lemma}\label{Grass almost-free} 
Let an algebraic group $G$ act on a projective space ${\proj}^n$ almost freely. 
If $a<n-{\dim}G$, then $G$ acts on $\mathbb{G}(a, {\proj}^n)$ almost freely. 
\end{lemma}

\begin{proof}
Let $p\in{\proj}^n$ be a general point. 
It suffices to show that a general $a$-plane $P\subset{\proj}^n$ through $p$ is not stabilized by any element of $G$. 
Consider the projection $\pi:{\proj}^n\dashrightarrow{\proj}^{n-1}$ from $p$. 
Since $\pi(G\cdot p\backslash p)\subset{\proj}^{n-1}$ has dimension $\leq{\dim}G$, 
a general $(a-1)$-plane $P'\subset{\proj}^{n-1}$ is disjoint from $\pi(G\cdot p\backslash p)$ 
by our assumption $a<n-{\dim}G$. 
This means that $P\cap(G\cdot p)=\{ p\}$. 
Now if $g\in G$ stabilizes $P$, we would have 
$g(p) = g(P\cap(G\cdot p)) =P\cap(G\cdot p) = p$, 
so that $g={\rm id}$. 
\end{proof}

\subsection{Representations of product groups}\label{ssec:prod grp} 

We can utilize quotients of Grassmannians for the rationality problem for representations of product groups 
(see \cite{Ma2} for more detail). 
Let $G, H$ be algebraic groups and $V, W$ be representations of $G, H$ respectively. 
Then $V\boxtimes W$ is a representation of $G\times H$. 
We assume that ${\dim}V<{\dim}W$. 
Identifying $V\boxtimes W$ with ${\hom}(V^{\vee}, W)$, we consider the map 
\begin{equation}\label{eqn:map to Grass}
{\hom}(V^{\vee}, W)\dashrightarrow G({\dim}V, W)
\end{equation} 
that sends a homomorphism to its image. 
Let $\mathcal{E}\to G({\dim}V, W)$ be the universal subbundle. 
Then \eqref{eqn:map to Grass} induces a birational map 
\begin{equation}\label{eqn:Grass method for boxtensor}
V\boxtimes W \dashrightarrow V\otimes\mathcal{E}
\end{equation}
to the vector bundle $V\otimes\mathcal{E}={\hom}(V^{\vee}, \mathcal{E})$ over $G({\dim}V, W)$. 
Here $H$ acts on $\mathcal{E}$ equivariantly and $G$ acts on $V$ fiberwisely. 
As in \S \ref{ssec:Grass quot}, let $H_0={\ker}(H\to{\PGL}(W))$ and $\bar{H}=H/H_0$. 
By the no-name method we then obtain 

\begin{lemma}[\cite{Ma2}]\label{Grass method for boxtensor}
Suppose that $\bar{H}$ acts on $G({\dim}V, W)$ almost freely and that 
$H_0$ acts on $\mathcal{E}\otimes({\det}\mathcal{E})^d$ trivially for some $d\in{\Z}$. 
Then 
\begin{equation*}
{\proj}(V\boxtimes W)/G\times H \sim ({\proj}(V^{\oplus{\dim}V})/G) \times (G({\dim}V, W)/H). 
\end{equation*}
\end{lemma}


\section{The case of ${\proj}^1\times{\proj}^2$}\label{sec:P1P2}

We begin the proof of Theorem \ref{main} with the cases $g\equiv6, 9\; (12)$, 
where the basic ${\proj}^2$-bundle is ${\proj}^1\times{\proj}^2$. 
We shall use the standard notation ${\OPP}(b, a)$ for line bundles on ${\proj}^1\times{\proj}^2$, rather than $L_{a,b}$ in \S \ref{sec:scroll}. 
The automorphism group of ${\proj}^1\times{\proj}^2$ is ${\PGL}_2\times{\PGL}_3$. 
It is useful to consider also  ${\SL}_2\times{\SL}_3$, because any line bundle ${\OPP}(b, a)$ is ${\SL}_2\times{\SL}_3$-linearized. 
The natural projection ${\SL}_2\times{\SL}_3\to{\PGL}_2\times{\PGL}_3$ has kernel ${\Z}/6$ generated by 
\begin{equation*}\label{eqn:scalar SL2SL3}
\zeta = (-1, e^{2\pi i/3}) \in {\SL}_2\times{\SL}_3. 
\end{equation*}
This element acts on ${\OPP}(b, a)$ by multiplication by $(e^{\pi i/3})^{3b-2a}$. 

In the sequel of this section let us use the abbreviation $W_a$ for the ${\SL}_3$-representation $H^0({\Oplane}(a))$. 
Thus we have $H^0({\OPP}(b, a)) \simeq V_b\boxtimes W_a$ as an ${\SL}_2\times{\SL}_3$-representation.

\subsection{The case $g\equiv6\; (12)$}\label{ssec:g6}

Let $b\geq3$ be an odd number. 
Let $\mathcal{L}\to{\proj}(V_b\boxtimes W_2)$ be the tautological bundle, and 
$\mathcal{E}\to{\proj}(V_b\boxtimes W_2)$ be the bundle 
$\underline{V_{b+1}\boxtimes W_2}/\mathcal{L}\otimes V_1$ as defined in Proposition \ref{tetra loci}. 
$\mathcal{L}$ and $\mathcal{E}$ are ${\SL}_2\times{\SL}_3$-linearized 
where $\zeta$ acts by $(e^{\pi i/3})^{3b-4}$ and $(e^{\pi i/3})^{3b-1}$ respectively. 
Recall that by Proposition \ref{tetra loci}, ${\proj}\mathcal{E}/{\SL}_2\times{\SL}_3$ is birational to the tetragonal locus of genus $6b\equiv6\;(12)$. 

\begin{lemma}\label{almost-free g6}
The group ${\PGL}_2\times{\PGL}_3$ acts on ${\proj}(V_b\boxtimes W_2)$ almost freely. 
\end{lemma}

\begin{proof}
It is known that a general member $S$ of $|{\OPP}(b, 2)|$ is the blow-up of ${\proj}^2$ at $3b+1$ points in general position 
(for example, put $s=3b$ and $n=s-1$ in \cite{dC-G} \S 2). 
Since $3b+1\geq10$, $S$ has no nontrivial automorphism (see \cite{Ko}). 
If $g\in{\PGL}_2\times{\PGL}_3$ acts trivially on $S$, so it does on ${\proj}^1\times{\proj}^2$. 
\end{proof}


\begin{proposition}\label{rational g6}
The quotient ${\proj}\mathcal{E}/{\SL}_2\times{\SL}_3$ is rational. 
\end{proposition}

\begin{proof}
Since $\zeta$ acts on $\mathcal{E}\otimes\mathcal{L}^2$ by multiplication by $(e^{\pi i/3})^{9b-9}=(-1)^{b-1}=1$, 
the bundle $\mathcal{E}\otimes\mathcal{L}^2$ is ${\PGL}_2\times{\PGL}_3$-linearized. 
We identify ${\proj}\mathcal{E}$ with ${\proj}(\mathcal{E}\otimes\mathcal{L}^2)$ and apply the no-name lemma to the latter, 
which is possible by the above lemma. 
Then we have 
\begin{equation*}
{\proj}\mathcal{E}/{\PGL}_2\times{\PGL}_3 \sim {\proj}^{6b+9}\times({\proj}(V_b\boxtimes W_2)/{\PGL}_2\times{\PGL}_3). 
\end{equation*}
In order to show that ${\proj}(V_b\boxtimes W_2)/{\SL}_2\times{\SL}_3$ is stably rational of level $6b+9$, 
we consider the product $U = {\proj}(V_b\boxtimes W_2) \times {\proj}(V_5\boxtimes W_1)$. 
Let $\mathcal{L}'\to{\proj}(V_5\boxtimes W_1)$ be the tautological bundle. 
The first projection $U\to{\proj}(V_b\boxtimes W_2)$ may be identified with 
the projective bundle ${\proj}(V_5\boxtimes W_1\otimes\mathcal{L})$, 
while the second $U\to{\proj}(V_5\boxtimes W_1)$ may be identified with ${\proj}(V_b\boxtimes W_2\otimes\mathcal{L}')$. 
Since $\zeta$ acts trivially on both $V_5\boxtimes W_1\otimes\mathcal{L}$ and $V_b\boxtimes W_2\otimes\mathcal{L}'$, 
these bundles are ${\PGL}_2\times{\PGL}_3$-linearized. 
Applying the no-name lemma to the two projections, we obtain 
\begin{eqnarray*}
U/{\PGL}_2\times{\PGL}_3 & \sim & {\proj}^{17}\times({\proj}(V_b\boxtimes W_2)/{\PGL}_2\times{\PGL}_3) \\  
                                          & \sim & {\proj}^{6b+5}\times({\proj}(V_5\boxtimes W_1)/{\PGL}_2\times{\PGL}_3). 
\end{eqnarray*}
Here ${\PGL}_2\times{\PGL}_3$ acts on ${\proj}(V_5\boxtimes W_1)$ almost freely because 
we have \eqref{eqn:Grass linear} and ${\PGL}_2$ acts on $G(3, V_5)$ almost freely (\cite{Ma2} Lemma 2.7).  
Thus the problem is reduced to stable rationality of level $6b+5$ of 
\begin{equation*}
{\proj}(V_5\boxtimes W_1)/{\SL}_2\times{\SL}_3 \sim G(3, V_5)/{\SL}_2. 
\end{equation*}
This in turn follows from Lemma \ref{Grass correspo} (1) and the rationality of ${\proj}V_5/{\SL}_2$ (which has dimension $2$). 
\end{proof}

\subsection{The ${\PGL}_2\times{\PGL}_3$-action on ${\proj}(V_1\boxtimes W_2)$}\label{ssec:V1W2}

Before going to the case $g\equiv9\;(12)$, 
we here study the action of ${\PGL}_2\times{\PGL}_3$ on ${\proj}(V_1\boxtimes W_2)$. 
Let $\mathcal{E}\to\mathbb{G}(1, {\proj}W_2)$ be the universal subbundle and  
consider the birational equivalence ${\proj}(V_1\boxtimes W_2)\sim {\proj}(V_1\otimes\mathcal{E})$ 
in \eqref{eqn:Grass method for boxtensor}. 
Then ${\PGL}_2$ acts on each fiber of ${\proj}(V_1\otimes\mathcal{E})\to\mathbb{G}(1, {\proj}W_2)$ almost freely and almost transitively. 
On the other hand, since a general conic pencil on ${\proj}^2$ is determined by its $4$ base points in general position, 
we see that ${\PGL}_3$ acts on $\mathbb{G}(1, {\proj}W_2)$ almost transitively 
and the stabilizer of a general pencil is the permutation group of its $4$ base points. 
Thus the ${\PGL}_2\times{\PGL}_3$-action on ${\proj}(V_1\boxtimes W_2)$ is almost transitive, 
with $\frak{S}_4$ the stabilizer of a general point. 
Let us study how this $\frak{S}_4$ acts on ${\proj}(V_1\boxtimes W_2)$.  

Recall first that the irreducible representations of $\frak{S}_4$ are the following five (\cite{Se}): 
\begin{itemize}
\item the trivial representation $\chi_0$; 
\item the sign representation $\epsilon$; 
\item the $3$-dimensional standard representation $\psi$; 
\item the tensor product $\epsilon\psi=\epsilon\otimes\psi$; and 
\item the $2$-dimensional standard representation $\theta$ of $\frak{S}_3$ 
where we regard $\frak{S}_3$ as the quotient of $\frak{S}_4$ by the Klein $4$-group. 
\end{itemize}

Now we may normalize the $4$ base points on ${\proj}^2$ so that ${\proj}^2={\proj}(\psi)$ as an $\frak{S}_4$-space. 
Then ${\proj}(W_2)={\proj}({\rm Sym}^2\psi^{\vee})$, and we have the decomposition  
\begin{equation*}
{\rm Sym}^2\psi^{\vee} \simeq {\rm Sym}^2\psi \simeq \chi_0\oplus \theta \oplus \psi. 
\end{equation*}
The conic pencil associated to the $4$ points is ${\proj}(\theta)\subset{\proj}W_2$. 
Since the fiber of $\mathcal{E}\to\mathbb{G}(1, {\proj}W_2)$ over the point ${\proj}(\theta)\in\mathbb{G}(1, {\proj}W_2)$ is $\theta$ itself, 
we see that ${\proj}V_1\simeq {\proj}(\theta^{\vee})$ as an $\frak{S}_4$-space. 
Hence ${\proj}(V_1\boxtimes W_2) \simeq {\proj}(\theta^{\vee}\otimes(\chi_0\oplus\theta\oplus\psi))$. 
Noticing that $\theta^{\vee}\simeq\theta$, $\theta^{\vee}\otimes\theta\simeq \chi_0\oplus\epsilon\oplus\theta$ 
and $\theta^{\vee}\otimes\psi\simeq \psi\oplus \epsilon\psi$, 
we summarize the argument as follows. 

\begin{lemma}\label{V1W2}
The group ${\PGL}_2\times{\PGL}_3$ acts on ${\proj}(V_1\boxtimes W_2)$ almost transitively, 
with the stabilizer of a general point $[v]$ isomorphic to $\frak{S}_4$. 
As an $\frak{S}_4$-space, 
\begin{equation}\label{eqn:V1W2 as S4space}
{\proj}(V_1\boxtimes W_2) \simeq {\proj}(\chi_0\oplus \epsilon\oplus \theta^{\oplus2}\oplus \psi\oplus \epsilon\psi)  
\end{equation}
with $[v]$ being ${\proj}(\chi_0)\in{\proj}(V_1\boxtimes W_2)$. 
\end{lemma}

For later use, we remark that the following fact is well-known: 
it is a simple application of the no-name method. 

\begin{proposition}\label{rational S4}
For any $\frak{S}_4$-representation $V$ the quotient ${\proj}V/\frak{S}_4$ is rational. 
\end{proposition}

\subsection{The case $g\equiv9\; (12)$}\label{ssec:g9}

Let $b>0$ be an even number. 
By Proposition \ref{tetra loci}, the quotient $G(2, V_b\boxtimes W_2)/{\SL}_2\times{\SL}_3$ is birational to 
the tetragonal locus of genus $6b-3\equiv9\;(12)$. 
We shall prove 

\begin{proposition}\label{rational g9}
The quotient $G(2, V_b\boxtimes W_2)/{\SL}_2\times{\SL}_3$ is rational for $b\ne2, 8$. 
\end{proposition}

To begin with, we rewrite $G(2, V_b\boxtimes W_2)/{\SL}_2\times{\SL}_3$ by \eqref{eqn:Grass linear} as 
\begin{equation}\label{Grass linear g9}
{\proj}(V_1\boxtimes V_b\boxtimes W_2)/{\SLSL}\times{\SL}_3. 
\end{equation}

\subsubsection{The case $b\geq12$}\label{sssec:b12}

When $b\geq12$, we have ${\dim}V_b>{\dim}(V_1\boxtimes W_2)=12$. 
We then want to use the method of \S \ref{ssec:prod grp} for  
$V_1\boxtimes V_b\boxtimes W_2 = (V_1\boxtimes W_2)\boxtimes V_b$, 
viewed as a representation of $({\SL}_2\times{\SL}_3)\times{\SL}_2$. 
Since $b$ is even, $V_b$ is a representation of ${\PGL}_2$ so that the universal subbundle over $G(12, V_b)$ is ${\PGL}_2$-linearized. 
By Lemma \ref{Grass almost-free}, ${\PGL}_2$ acts almost freely on $G(12, V_b)\simeq G(b-11, V_b)$ . 
Thus we can apply Lemma \ref{Grass method for boxtensor} to see that 
\eqref{Grass linear g9} is birational to 
\begin{equation*}
({\proj}(V_1\boxtimes W_2)^{\oplus12}/{\SL}_2\times{\SL}_3) \times (G(12, V_b)/{\SL}_2). 
\end{equation*}
Since ${\proj}V_b/{\SL}_2$ is rational by Katsylo \cite{Ka1}, 
$G(12, V_b)/{\SL}_2$ is stably rational of level $11$ by Lemma \ref{Grass correspo} (1). 
Hence it remains to prove that ${\proj}(V_1\boxtimes W_2)^{\oplus12}/{\SL}_2\times{\SL}_3$ is rational. 

By Lemma \ref{V1W2}, we can apply the slice method to the projection 
${\proj}(V_1\boxtimes W_2)^{\oplus12}\dashrightarrow{\proj}(V_1\boxtimes W_2)$ to the first summand. 
Then we have 
\begin{equation*}
{\proj}(V_1\boxtimes W_2)^{\oplus12}/{\PGL}_2\times{\PGL}_3 \sim (V_1\boxtimes W_2)^{\oplus11}/\frak{S}_4. 
\end{equation*}
The right side is rational by Proposition \ref{rational S4}. 
Thus Proposition \ref{rational g9} is proved for $b\geq12$.

\subsubsection{The case $b=4, 6$}\label{sssec:b46}

Let $b$ be either $4$ or $6$. 
Then ${\dim}V_b<{\dim}(V_1\boxtimes W_2)$. 
We shall consider $V_1\boxtimes V_b\boxtimes W_2$ as 
the ${\SL}_2\times({\SL}_2\times{\SL}_3)$-representation $V_b\boxtimes (V_1\boxtimes W_2)$, 
and apply the method of \S \ref{ssec:prod grp}.  
Let $\mathcal{E}\to G(b+1, V_1\boxtimes W_2)$ be the universal subbundle. 
Then $\zeta$ acts on $\mathcal{E}\otimes{\det}\mathcal{E}$ trivially in case $b=4$, 
and on $\mathcal{E}\otimes({\det}\mathcal{E})^{-1}$ trivially in case $b=6$. 
One checks (e.g., by looking at various special loci in ${\proj}(V_1\boxtimes W_2)$) that 
${\PGL}_2\times{\PGL}_3$ acts on $G(b+1, V_1\boxtimes W_2)$ almost freely. 
So by Lemma \ref{Grass method for boxtensor}, \eqref{Grass linear g9} is birational to 
\begin{equation*}
({\proj}V_b^{\oplus b+1}/{\SL}_2)\times(G(b+1, V_1\boxtimes W_2)/{\SL}_2\times{\SL}_3). 
\end{equation*}
The first quotient ${\proj}V_b^{\oplus b+1}/{\SL}_2$ is rational by Katsylo \cite{Ka2}, 
so it suffices to show that $G(b+1, V_1\boxtimes W_2)/{\SL}_2\times{\SL}_3$ is stably rational of level $(b+1)^2-4$. 

We regard $V_1\boxtimes W_2$ as an $\frak{S}_4$-representation as in the right side of \eqref{eqn:V1W2 as S4space}  
and denote $V=V_1\boxtimes W_2/\chi_0$. 
Combining Lemma \ref{Grass correspo} (2) and Lemma \ref{V1W2}, we see that 
\begin{equation*}
{\proj}^b\times(G(b+1, V_1\boxtimes W_2)/{\SL}_2\times{\SL}_3) \sim G(b, V)/\frak{S}_4. 
\end{equation*}
By looking at the decomposition \eqref{eqn:V1W2 as S4space}, 
we can find an $\frak{S}_4$-invariant subspace $V'\subset V$ of dimension $b$ in either case. 
If $V''\subset V$ is the complementary sub $\frak{S}_4$-representation, 
we have the $\frak{S}_4$-invariant open set ${\hom}(V', V'')\subset G(b, V)$ 
where $\frak{S}_4$ acts on ${\hom}(V', V'')$ linearly. 
Then ${\hom}(V', V'')/\frak{S}_4$ is rational by Proposition \ref{rational S4}, 
so Proposition \ref{rational g9} is proved for $b=4, 6$.

\begin{remark}
It seems that the same approach does not work for $b=2, 8$, 
because $\zeta$ acts on $\mathcal{E}\otimes({\det}\mathcal{E})^d$ nontrivially for any $d\in{\Z}$. 
\end{remark}

\subsubsection{The case $b=10$}\label{sssec:b10}

As in \S \ref{sssec:b46}, we consider $V_1\boxtimes V_{10}\boxtimes W_2$ as 
the ${\SL}_2\times({\SL}_2\times{\SL}_3)$-representation $V_{10}\boxtimes (V_1\boxtimes W_2)$  
and identify it with the vector bundle $V_{10}\otimes\mathcal{E}$ over $G(11, V_1\boxtimes W_2)$, 
where $\mathcal{E}$ is the universal subbundle. 
In this case, we have 
\begin{equation*}
G(11, V_1\boxtimes W_2) = {\proj}(V_1\boxtimes W_2)^{\vee} = {\proj}(V_1^{\vee}\boxtimes W_2^{\vee}). 
\end{equation*}
By identifying ${\GL}_n={\GL}({\C}^n)$ with ${\GL}(({\C}^{n})^{\vee})$ through the dual representation, 
we can apply the result of \S \ref{ssec:V1W2} to the ${\PGL}_2\times{\PGL}_3$-action on ${\proj}(V_1^{\vee}\boxtimes W_2^{\vee})$. 
Thus we find that it is almost transitive with 
the stabilizer of a general point $[H]\in{\proj}(V_1\boxtimes W_2)^{\vee}$ isomorphic to $\frak{S}_4$, 
and the corresponding hyperplane $H\subset {\proj}(V_1\boxtimes W_2)$ is isomorphic to 
\begin{equation*}
{\proj}(\epsilon\oplus\theta^{\oplus2}\oplus\psi\oplus\epsilon\psi)^{\vee} \simeq {\proj}(\epsilon\oplus\theta^{\oplus2}\oplus\psi\oplus\epsilon\psi) 
\end{equation*}
as an $\frak{S}_4$-space. 
We set $V=\epsilon\oplus\theta^{\oplus2}\oplus\psi\oplus\epsilon\psi$. 

We apply the slice method to the projection 
${\proj}(V_{10}\otimes\mathcal{E})\to{\proj}(V_1\boxtimes W_2)^{\vee}$. 
This gives  
\begin{equation*}
{\proj}(V_{10}\otimes\mathcal{E})/{\PGL}_2\times{\PGL}_2\times{\PGL}_3 
\sim {\proj}(V_{10}\boxtimes V)/{\PGL}_2\times\frak{S}_4. 
\end{equation*}
Next we use the no-name method for the projection ${\proj}(V_{10}\boxtimes V)\dashrightarrow{\proj}(V_{10}\boxtimes\psi)$ 
from the rest summand $V_{10}\boxtimes(\epsilon\oplus \theta^{\oplus2}\oplus\epsilon\psi)$. 
Then we have  
\begin{equation*}
{\proj}(V_{10}\boxtimes V)/{\PGL}_2\times\frak{S}_4 \sim {\C}^{88}\times ({\proj}(V_{10}\boxtimes\psi)/{\PGL}_2\times\frak{S}_4). 
\end{equation*}
Finally, we apply Lemma \ref{Grass method for boxtensor} to the ${\PGL}_2\times\frak{S}_4$-representation $V_{10}\boxtimes\psi$. 
The group ${\PGL}_2$ acts on $G(3, V_{10})$ almost freely by Lemma \ref{Grass almost-free}. 
Then the quotient $G(3, V_{10})/{\SL}_2$ is stably rational of level $2$ 
by Lemma \ref{Grass correspo} (1) and the rationality of ${\proj}V_{10}/{\SL}_2$ (\cite{B-K}).  
On the other hand, ${\proj}(\psi^{\oplus3})/\frak{S}_4$ is rational by Proposition \ref{rational S4}. 
Hence by Lemma \ref{Grass method for boxtensor} we conclude that 
${\proj}(V_{10}\boxtimes\psi)/{\PGL}_2\times\frak{S}_4$ is rational. 
This finishes the proof of Proposition \ref{rational g9} for $b=10$.


\section{The case of the blown-up ${\proj}^3$}\label{sec:P3}

In this section we study the cases $g\equiv1, 10\; (12)$ in Theorem \ref{main}, 
where the basic ${\proj}^2$-bundle is $X_{0,1}$, the blow-up of ${\proj}^3$ along a line $l$. 
We keep the notation in \S \ref{ssec:X01}.

\subsection{The case $g\equiv10\; (12)$}\label{ssec:g10}

Let $b>0$ be an odd number. 
Let $\mathcal{L}\to|L_{2,b}|$ be the tautological bundle and $\mathcal{E}\to|L_{2,b}|$ be the bundle 
$\underline{H^0(L_{2,b+1})}/\mathcal{L}\otimes H^0(L_{0,1})$ as defined in Proposition \ref{tetra loci}. 
Since $L_{2,b+1}$ is ${\aut}(X_{0,1})$-linearized by Lemma \ref{linearization X01}, $\mathcal{E}$ is ${\aut}(X_{0,1})$-linearized. 
The quotient ${\proj}\mathcal{E}/{\aut}(X_{0,1})$ is birational to 
the tetragonal locus of genus $6b+4\equiv10\;(12)$ by Proposition \ref{tetra loci}.

\begin{lemma}\label{almost-free g10}
The group ${\aut}(X_{0,1})$ acts on $|L_{2,b}|$ almost freely. 
\end{lemma}

\begin{proof}
As in Lemma \ref{almost-free g6}, a general member of $|L_{2,b}|$ is the blow-up of ${\proj}^2$ at $3b+3$ general points 
(put $s=3b+2$, $n=s-1$ in \cite{dC-G} \S 2), 
and such a surface has no nontrivial automorphism: see \cite{Ko} for the case $b\geq2$, while the case $b=1$ is well-known. 
(See also Lemma \ref{almost free g1} for another approach.)   
\end{proof}

\begin{proposition}\label{rational g10}
The quotient ${\proj}\mathcal{E}/{\aut}(X_{0,1})$ is rational. 
\end{proposition}

\begin{proof}
By the no-name lemma we see that 
\begin{equation*}
{\proj}\mathcal{E}/{\aut}(X_{0,1}) \sim {\proj}^{6b+13}\times(|L_{2,b}|/{\aut}(X_{0,1})), 
\end{equation*}
so it suffices to show that $|L_{2,b}|/{\aut}(X_{0,1})$ is stably rational of level $6b+13$. 
Consider the product $U=|L_{2,b}|\times|L_{2,1}|$. 
We can identify the first projection $U\to|L_{2,b}|$ with the projective bundle ${\proj}(\mathcal{L}\otimes H^0(L_{2,1}))$, 
and the second $U\to|L_{2,1}|$ with ${\proj}(\mathcal{L}'\otimes H^0(L_{2,b}))$ where $\mathcal{L}'\to|L_{2,1}|$ is the tautological bundle. 
Since $L_{4,b+1}$ is ${\aut}(X_{0,1})$-linearized by Lemma \ref{linearization X01}, 
both $\mathcal{L}\otimes H^0(L_{2,1})$ and $\mathcal{L}'\otimes H^0(L_{2,b})$ are ${\aut}(X_{0,1})$-linearized. 
Then we obtain by the no-name method that 
\begin{eqnarray*}
U/{\aut}(X_{0,1}) & \sim & {\proj}^{15}\times(|L_{2,b}|/{\aut}(X_{0,1})) \\ 
& \sim & {\proj}^{6b+9}\times(|L_{2,1}|/{\aut}(X_{0,1})). 
\end{eqnarray*}

We shall prove that $|L_{2,1}|/{\aut}(X_{0,1})$ is stably rational of level $1$. 
Recall that 
$|L_{2,1}|$ is identified with the linear system of cubic surfaces in ${\proj}^3$ containing the line $l$. 
Thus, if we consider the parameter space 
\begin{equation*}
V = \{ (S, l')\in|{\Ospace}(3)|\times\mathbb{G}(1, {\proj}^3) \; | \; l'\subset S \}, 
\end{equation*}
then $|L_{2,1}|/{\aut}(X_{0,1})$ gets bitational to $V/{\PGL}_4$, the moduli space of cubic surfaces with a line on it. 
Let $\mathcal{F}\to V$ be the pullback of the universal subbundle over $\mathbb{G}(1, {\proj}^3)$. 
We have 
\begin{equation*}
{\proj}\mathcal{F} = \{ (S, l', p)\in|{\Ospace}(3)|\times\mathbb{G}(1, {\proj}^3)\times{\proj}^3 \; | \; p\in l'\subset S \}. 
\end{equation*}
Let $\mathcal{F}'$ be the twist of $\mathcal{F}$ by the pullback of ${\sheaf}_{|{\Ospace}(3)|}(1)$. 
We can identify ${\proj}\mathcal{F}$ with ${\proj}\mathcal{F}'$,  
and $\mathcal{F}'$ is ${\PGL}_4$-linearized because $\sqrt{-1}\in{\SL}_4$ acts on it trivially. 
By the no-name lemma for $\mathcal{F}'$ we have 
\begin{equation*}
{\proj}\mathcal{F}/{\PGL}_4 \sim {\proj}^1\times(V/{\PGL}_4). 
\end{equation*}
On the other hand, consider the space $T$ of flags $p\in l'\subset H\subset{\proj}^3$, where $H$ is a plane. 
We have the ${\PGL}_4$-equivariant map 
\begin{equation}\label{eqn:map to flag}
{\proj}\mathcal{F} \dashrightarrow T, \qquad (S, l', p)\mapsto(p\in l'\subset T_pS). 
\end{equation}
Its fiber over $(p\in l'\subset H)\in T$ is an open set of a \textit{linear} system ${\proj}W$ in $|{\Ospace}(3)|$. 
The group ${\GL}_4$ acts on $T$ transitively with the stabilizer $G$ of $(p\in l'\subset H)$ being connected and solvable. 
By the slice method for \eqref{eqn:map to flag} we see that 
\begin{equation*}
{\proj}\mathcal{F}/{\PGL}_4 \sim {\proj}W/G, 
\end{equation*}
and ${\proj}W/G$ is rational by Miyata's theorem \cite{Mi}. 
Hence $|L_{2,1}|/{\aut}(X_{0,1})$ is stably rational of level $1$. 
This finishes the proof of Proposition \ref{rational g10}. 
\end{proof}

\subsection{The case $g\equiv1\; (12)$}\label{ssec:g1}

Let $b>0$ be an even number. 
Then $L_{2,b}$ is ${\aut}(X_{0,1})$-linearized by Lemma \ref{linearization X01}. 
The quotient $\mathbb{G}(1, |L_{2,b}|)/{\aut}(X_{0,1})$ is birational to the tetragonal locus of genus $6b+1\equiv1\;(12)$. 

Before proving its rationality, 
we recall that the ${\aut}(X_{0,1})$-representation $H^0(L_{2,b})$ is reducible: 
it has the invariant filtration 
\begin{equation}\label{filtration X01}
0 \subset H^0(L_{2,-2})\otimes H^0(L_{0,b+2}) \subset H^0(L_{1,-1})\otimes H^0(L_{1,b+1}) \subset H^0(L_{2,b}) 
\end{equation}
defined by the vanishing orders along $\Sigma$. 
Here $H^0(L_{d,-d})$ is $1$-dimensional and defines $d\Sigma$. 
If we consider $H^0(L_{2,b})$ as a representation of the double cover 
\begin{equation*}
\tilde{G}={\SL}_2\ltimes {\hom}(V_1, V_1)\rtimes{\GL}_2
\end{equation*}
of ${\aut}(X_{0,1})$, the successive quotients of \eqref{filtration X01} 
are the ${\SL}_2\times{\GL}_2$-representations 
\begin{equation*}
V_{b+2,0}, \quad V_{b+1,1}, \quad V_{b,2}, 
\end{equation*}
with the action of ${\hom}(V_1, V_1)$ as described in \eqref{eqn:unipotent radical action X01 general}. 
This structure of $H^0(L_{2,b})$ was first observed in case $b=1$ by B\"ohning-Bothmer-Casnati \cite{B-B-C}. 

We consider the quotient representation  
\begin{equation*}
W = H^0(L_{2,b})/(H^0(L_{2,-2})\otimes H^0(L_{0,b+2})). 
\end{equation*}
Geometrically the quotient map $H^0(L_{2,b})\to W$ gives 
the $\leq1$-th Taylor development of the sections of $L_{2,b}$ along $\Sigma$. 

\begin{lemma}\label{almost free g1}
The group ${\aut}(X_{0,1})$ acts on ${\proj}W={\proj}(V_{b+1,1}\oplus V_{b,2})$ almost freely. 
\end{lemma}

\begin{proof}
Suppose that for a general point $[F_1, F_2]\in{\proj}(V_{b+1,1}\oplus V_{b,2})$ 
we have an element $(g_1, h, g_2)\in\tilde{G}$ fixing it. 
Consider the projection ${\proj}W\dashrightarrow{\proj}V_{b,2}$ from $V_{b+1,1}$ which is $\tilde{G}$-equivariant. 
Since ${\PGL}_2\times{\PGL}_2$ acts on ${\proj}V_{b,2}$ almost freely, 
we must have $(g_1, g_2)=(\pm1, \lambda)$ for some scalar $\lambda\in{\C}^{\times}$. 
Composing it with $(-1, -1)\in{\SL}_2\times{\GL}_2$, we may assume $g_1=1$. 
Now $(1, h, \lambda)$ maps $[F_1, F_2]$ to $[\lambda^{-1}F_1+\lambda^{-1}\langle F_2, h\rangle, F_2]$, 
where $\langle F_2, \cdot \rangle:{\hom}(V_1, V_1) \to V_{b+1,1}$ is the linear map 
induced by the multiplication and the contraction. 
Thus we have $(\lambda-1)F_1=\langle F_2, h\rangle$. 
Since the map $\langle F_2, \cdot \rangle$ is injective for general $F_2\in V_{b,2}$, 
choosing $F_1$ generically we have $\lambda=1$ and $h=0$. 
\end{proof}

Now we prove 

\begin{proposition}\label{rational g1}
The quotient $\mathbb{G}(1, |L_{2,b}|)/{\aut}(X_{0,1})$ is rational. 
\end{proposition}

\begin{proof}
In the first step, we apply Lemma \ref{Grass no-name} to the quotient map $H^0(L_{2,b})\to W$. 
By Lemma \ref{almost free g1} and Lemma \ref{Grass almost-free}, 
${\aut}(X_{0,1})$ acts on $\mathbb{G}(1, {\proj}W)$ almost freely. 
Then we obtain 
\begin{equation*}
\mathbb{G}(1, |L_{2,b}|)/{\aut}(X_{0,1}) \sim {\C}^{2b+6}\times(\mathbb{G}(1, {\proj}W)/{\aut}(X_{0,1})). 
\end{equation*}
By Lemma \ref{Grass correspo} (1) we have   
\begin{equation*}
{\proj}^1\times(\mathbb{G}(1, {\proj}W)/{\aut}(X_{0,1})) \sim {\proj}^{5b+5}\times({\proj}W/{\aut}(X_{0,1})). 
\end{equation*}
We shall use the no-name method for ${\proj}W\times(|L_{2,0}|\times|L_{1,1}|)$. 
Since both $L_{2,0}$ and $L_{1,1}$ are ${\aut}(X_{0,1})$-linearized and 
since ${\aut}(X_{0,1})$ acts on $|L_{2,0}|\times|L_{1,1}|$ almost freely, we see that 
\begin{eqnarray*}
({\proj}W\times|L_{2,0}|\times|L_{1,1}|)/{\aut}(X_{0,1}) 
& \sim & {\proj}^9\times{\proj}^6\times({\proj}W/{\aut}(X_{0,1})) \\ 
& \sim & {\proj}^{5b+6}\times((|L_{2,0}|\times|L_{1,1}|)/{\aut}(X_{0,1})). 
\end{eqnarray*}
In this way, we are reduced to showing that $(|L_{2,0}|\times|L_{1,1}|)/{\aut}(X_{0,1})$ is stably rational of level $5b+6$. 
Actually, we shall prove that it is rational. 

We identify ${\aut}(X_{0,1})$ with the stabilizer in ${\PGL}_4$ of the line $l$, 
$|L_{2,0}|$ with $|{\Ospace}(2)|$, and $|L_{1,1}|$ with the linear system of quadrics containing $l$. 
This implies that ${\aut}(X_{0,1})$ acts on $|L_{1,1}|$ almost transitively, 
with the stabilizer of a general $Q\in|L_{1,1}|$ isomorphic to $({\C}^{\times}\ltimes{\C})\times{\PGL}_2$ 
(which is the stabilizer of $l$ in ${\aut}(Q)$). 
By the slice method for the projection $|L_{2,0}|\times|L_{1,1}|\to|L_{1,1}|$ we obtain 
\begin{eqnarray*}
(|L_{2,0}|\times|L_{1,1}|)/{\aut}(X_{0,1}) 
& \sim & |L_{2,0}|/({\C}^{\times}\ltimes{\C})\times{\PGL}_2 \\ 
& \sim & {\C}\times (|{\OQ}(2, 2)|/({\C}^{\times}\ltimes{\C})\times{\PGL}_2). 
\end{eqnarray*}
Consider the product $U=|{\OQ}(1, 0)|\times|{\OQ}(2, 2)|$. 
Note that ${\OQ}(1, 0)$ and ${\OQ}(2, 2)$ are both $({\C}^{\times}\ltimes{\C})\times{\PGL}_2$-linearized. 
By the no-name lemma for the second projection $U\to|{\OQ}(2, 2)|$ we have 
\begin{equation*}
U/({\C}^{\times}\ltimes{\C})\times{\PGL}_2 \sim {\proj}^1\times(|{\OQ}(2, 2)|/({\C}^{\times}\ltimes{\C})\times{\PGL}_2). 
\end{equation*}
On the other hand, using the slice method for the first projection $U\to|{\OQ}(1, 0)|$, 
we deduce that 
\begin{equation*}
U/({\C}^{\times}\ltimes{\C})\times{\PGL}_2 \sim |{\OQ}(2, 2)|/{\C}^{\times}\times{\PGL}_2. 
\end{equation*}
The last quotient is rational by Katsylo \cite{Ka2}. 
Thus $(|L_{2,0}|\times|L_{1,1}|)/{\aut}(X_{0,1})$ is rational, and the proof of the proposition is completed. 
\end{proof}


\section{The case of the small resolution of quadric cone}\label{sec:quadric cone}

In this section we study the cases $g\equiv2, 5\; (12)$ in Theorem \ref{main}, 
where the basic ${\proj}^2$-bundle is $X_{1,1}$, a small resolution of a quadric cone. 
We use the notation in \S \ref{ssec:X11} freely.

\subsection{The case $g\equiv2\; (12)$}\label{ssec:g2}

Let $b>0$ be an odd number. 
Let $\mathcal{L}\to|L_{2,b}|$ be the tautological bundle and $\mathcal{E}\to|L_{2,b}|$ be the bundle 
$\underline{H^0(L_{2,b+1})}/\mathcal{L}\otimes H^0(L_{0,1})$ as defined in Proposition \ref{tetra loci}. 
Then ${\proj}\mathcal{E}/{\aut}(X_{1,1})$ is birational to the tetragonal locus of genus $6b+8\equiv2\;(12)$. 
Note that $\mathcal{E}$ is ${\aut}(X_{1,1})$-linearized because $L_{2,b+1}$ is so by Lemma \ref{linearization X11}. 

\begin{lemma}\label{almost-free g2}
The group ${\aut}(X_{1,1})$ acts on $|L_{2,b}|$ almost freely. 
\end{lemma}

\begin{proof}
As in Lemma \ref{almost-free g6}, a general member $S$ of $|L_{2,b}|$ is the blow-up of ${\proj}^2$ at $3b+5$ general points 
(put $s=3b+4$, $n=s-1$ in \cite{dC-G} \S 2). 
When $b\geq2$, $S$ has no nontrivial automorphism (\cite{Ko}). 
On the other hand, when $b=1$, $S$ has the Geiser involution 
but this does not preserve the line bundle giving the embedding $S\subset X_{1,1}$.    
\end{proof}

\begin{proposition}\label{rational g2}
The quotient ${\proj}\mathcal{E}/{\aut}(X_{1,1})$ is rational. 
\end{proposition}

\begin{proof}
By the no-name lemma we have 
\begin{equation*}
{\proj}\mathcal{E}/{\aut}(X_{1,1}) \sim {\proj}^{6b+17}\times(|L_{2,b}|/{\aut}(X_{1,1})). 
\end{equation*}
To deduce stable rationality of $|L_{2,b}|/{\aut}(X_{1,1})$,  
we consider the product $U=|L_{2,b}|\times(|L_{0,1}|\times|L_{1,-1}|)$. 
We can identify the projection $|L_{2,b}|\times|L_{0,1}|\to|L_{2,b}|$ with 
the projective bundle ${\proj}(\mathcal{L}\otimes H^0(L_{0,1}))$, 
and $|L_{2,b}|\times|L_{1,-1}|\to|L_{2,b}|$ with ${\proj}(\mathcal{L}\otimes H^0(L_{1,-1}))$. 
Note that $\mathcal{L}\otimes H^0(L_{0,1})$ and $\mathcal{L}\otimes H^0(L_{1,-1})$ are 
${\aut}(X_{1,1})$-linearized by Lemma \ref{linearization X11}. 
Then by the no-name lemma we obtain 
\begin{equation*}
U/{\aut}(X_{1,1}) \sim {\proj}^1\times{\proj}^1\times(|L_{2,b}|/{\aut}(X_{1,1})). 
\end{equation*}
On the other hand, we use the slice method for the projection $U\to|L_{0,1}|\times|L_{1,-1}|$. 
Since $|L_{0,1}|\times|L_{1,-1}|$ is identified with the base quadric $Q_0$ of the quadric cone, 
we see that $({\SL}_2\times{\GL}_2)\ltimes V_{1,1}$ acts on it transitively. 
If $G_1$ is the stabilizer in ${\SL}_2\times{\GL}_2$ of a point $p$ of $Q_0$, 
then 
we have 
\begin{equation*}
U/{\aut}(X_{1,1}) \sim |L_{2,b}|/G_1\ltimes V_{1,1}. 
\end{equation*}
Since $G_1$ is connected and solvable, so is $G_1\ltimes V_{1,1}$. 
Thus $|L_{2,b}|/G_1\ltimes V_{1,1}$ is rational by Miyata's theorem \cite{Mi}. 
This shows that $|L_{2,b}|/{\aut}(X_{1,1})$ is stably rational of level $2$, 
and the proposition is proved. 
\end{proof}

\subsection{The case $g\equiv5\; (12)$}\label{ssec:g5}

Let $b>0$ be an even number. 
By Proposition \ref{tetra loci}, the quotient $\mathbb{G}(1, |L_{2,b}|)/{\aut}(X_{1,1})$ is birational to 
the tetragonal locus of genus $6b+5\equiv5\;(12)$. 
Here $L_{2,b}$ is ${\aut}(X_{1,1})$-linearized by Lemma \ref{linearization X11}. 

Let $F$ be the kernel of the restriction map $H^0(L_{2,b})\to H^0({\sheaf}_{\sigma}(b))$. 
Recall that the ${\aut}(X_{1,1})$-representation $H^0(L_{2,b})$ is reducible, 
having the invariant filtration 
\begin{equation}\label{filter X11}
0\subset f^{\ast}H^0({\OQQ}(b+2, 2)) \subset F \subset H^0(L_{2,b})
\end{equation}
defined by the vanishing orders along $\sigma$. 
If we consider $H^0(L_{2,b})$ as a representation of the double cover  
$({\SL}_2\times{\GL}_2)\ltimes V_{1,1}$ of ${\aut}(X_{1,1})$, 
the successive quotients of \eqref{filter X11} are the ${\SL}_2\times{\GL}_2$-representations 
\begin{equation*}
V_{b+2,2}, \quad V_{b+1,1}, \quad V_{b,0}, 
\end{equation*}
and the unipotent radical $V_{1,1}\ni h$ acts by multiplication by ${\exp}(h)$, 
\begin{equation*}
V_{b+1,1} \to V_{b+1,1}\oplus V_{b+2,2}, \qquad V_{b,0}\to V_{b,0}\oplus V_{b+1,1}\oplus V_{b+2,2}. 
\end{equation*}

We consider the quotient representation  
\begin{equation*}
W = H^0(L_{2,b})/ f^{\ast}H^0({\OQQ}(b+2, 2)). 
\end{equation*}
Let $([X_0, X_1], [Y_0, Y_1])$ be bi-homogeneous coordinates of $Q_0\simeq{\proj}^1\times{\proj}^1$.

\begin{lemma}\label{almost free g5}
The following holds. 

$(1)$ When $b\geq4$, ${\aut}(X_{1,1})$ acts on ${\proj}W$ and $\mathbb{G}(1, {\proj}W)$ almost freely. 

$(2)$ When $b=2$, ${\aut}(X_{1,1})$ acts on $\mathbb{G}(1, {\proj}W)$ almost freely. 

$(3)$ When $b=2$, ${\aut}(X_{1,1})$ acts on ${\proj}W={\proj}(V_{3,1}\oplus V_{2,0})$ almost transitively. 
If $v=(X_0^3Y_1+X_1^3Y_0, X_0X_1)\in W$, the stabilizer of $[v]\in{\proj}W$ is the subgroup 
\begin{equation*}
\frak{S}_2\ltimes{\C}^{\times}\subset{\SL}_2\times{\GL}_2/(-1, -1)
\end{equation*}
generated by 
\begin{equation*}
\frak{S}_2 : X_0\mapsto X_1, \quad X_1\mapsto -X_0, \quad Y_0\mapsto Y_1, \quad Y_1\mapsto -Y_0, 
\end{equation*}
\begin{equation}\label{eqn:torus action g=17}
\alpha\in{\C}^{\times} : X_0\mapsto\alpha X_0,\quad X_1\mapsto\alpha^{-1}X_1,\quad Y_0\mapsto\alpha^3Y_0,\quad Y_1\mapsto\alpha^{-3}Y_1.  
\end{equation}
\end{lemma}

\begin{proof}
(1) 
In view of Lemma \ref{Grass almost-free}, we prove the assertion only for ${\proj}W={\proj}(V_{b+1,1}\oplus V_{b,0})$. 
Let $K\subset{\SL}_2$ be the stabilizer of a general point $[F]\in{\proj}V_{b,0}$. 
It suffices to show that $(K\times{\GL}_2)\ltimes V_{1,1}$ modulo $(-1, -1)$ acts on $({\C}F)^{\vee}\otimes V_{b+1,1}$ almost freely. 
Here $V_{1,1}$ acts as translation by $F^{\vee}\otimes(V_{1,1}\cdot F)$. 
Consider the quotient map 
\begin{equation}\label{eqn:X11 auxi VB}
({\C}F)^{\vee}\otimes V_{b+1,1}\to ({\C}F)^{\vee}\otimes V_{b+1,1}/(({\C}F)^{\vee}\otimes(V_{1,1}\cdot F)). 
\end{equation}
This is a $K\times{\GL}_2$-linearized vector bundle on which $V_{1,1}$ acts by translations in the fibers (in particular, freely).  
If we set $U=({\C}F)^{\vee}\otimes (V_{b+1}/V_1\cdot F)$, 
the image of \eqref{eqn:X11 auxi VB} is the $K\times{\GL}_2$-representation $U\boxtimes V_1$. 
It is easy to see that $K\times{\GL}_2/(-1, -1)$ acts on $U\boxtimes V_1$ almost freely: 
when $b\geq6$, we have $K=\{\pm1\}$ and ${\GL}_2$ acts on $V_1^{\oplus b}$ almost freely. 
When $b=4$, $K/\pm1$ is the Klein $4$-group which acts on $G(2, U)$ effectively. 
Then our assertion follows by considering the fibration $U\boxtimes V_1\dashrightarrow G(2, U)$ as in \eqref{eqn:map to Grass}. 

(2) 
A general $2$-dimensional linear subspace of $W$ can be normalized by the ${\aut}(X_{1,1})$-action to the following type: 
\begin{equation*}
P = \langle (F_0, X_0^2), \; (F_1, X_1^2) \rangle, \quad F_i\in V_{3,1}. 
\end{equation*}
The basis presented here is canonical, in that 
its image by the projection $\pi: W\to V_{2,0}$ gives the discriminant locus of the conic pencil $\pi(P)$. 
Hence by any stabilizer of $P$ this basis is preserved up to scalar.  
Using this property, our assertion follows from a direct calculation. 

(3) 
We only have to determine the stabilizer. 
Clearly the group $\frak{S}_2\ltimes{\C}^{\times}$ defined above fixes $[v]$. 
Conversely, suppose $g\in{\aut}(X_{1,1})$ fixes $[v]$. 
Composing $g$ with an element of $\frak{S}_2\ltimes{\C}^{\times}$, we may assume that 
$g$ is the projection image of an element of the form $(1, g_2, h)\in({\SL}_2\times{\GL}_2)\ltimes V_{1,1}$. 
Then we would have 
\begin{equation*}
X_0^3(g_2(Y_1)-Y_1) + X_1^3(g_2(Y_0)-Y_0) = -hX_0X_1,
\end{equation*}
from which follow $h=0$ and $g_2=1$. 
\end{proof}

Now we prove 

\begin{proposition}\label{rational g5}
The quotient $\mathbb{G}(1, |L_{2,b}|)/{\aut}(X_{1,1})$ is rational. 
\end{proposition}

\begin{proof}
By Lemma \ref{almost free g5} (1) and (2), we can apply Lemma \ref{Grass no-name} to the quotient homomorphism $H^0(L_{2,b})\to W$. 
Then we obtain 
\begin{equation*}
\mathbb{G}(1, |L_{2,b}|)/{\aut}(X_{1,1}) \sim {\C}^{6b+18}\times(\mathbb{G}(1, {\proj}W)/{\aut}(X_{1,1})). 
\end{equation*}

In case $b\geq4$, we can use Lemma \ref{Grass correspo} (1) to see that 
\begin{equation*}
{\proj}^1\times(\mathbb{G}(1, {\proj}W)/{\aut}(X_{1,1})) \sim {\proj}^{3b+3}\times({\proj}W/{\aut}(X_{1,1})), 
\end{equation*}
so we are reduced to proving stable rationality of ${\proj}W/{\aut}(X_{1,1})$ of level $3b+3$. 
We shall utilize the duality ${\aut}(X_{1,1})\simeq{\aut}(X_{0,1})$ from Lemma \ref{duality}. 
In the proof of Proposition \ref{rational g1}, 
we have found representations $U_1$, $U_2$ of ${\aut}(X_{0,1})$ of dimension $10$, $7$ such that 
${\aut}(X_{0,1})$ acts on ${\proj}U_1\times{\proj}U_2$ almost freely with the quotient rational. 
Replacing ${\aut}(X_{0,1})$ with ${\aut}(X_{1,1})$,  
we can repeat the same no-name argument for ${\proj}W\times({\proj}U_1\times{\proj}U_2)$  
to deduce stable rationality of ${\proj}W/{\aut}(X_{1,1})$ of level $15$. 
This proves our assertion for $b\geq4$. 
   
Next we consider the case $b=2$. 
Let $v\in W$ be the vector as defined in Lemma \ref{almost free g5} (3). 
By Lemma \ref{almost free g5} (2), (3) and Lemma \ref{Grass correspo} (2), we have 
\begin{equation*}
{\proj}^1\times(\mathbb{G}(1, {\proj}W)/{\aut}(X_{1,1})) \sim {\proj}(W/{\C}v)/\frak{S}_2\ltimes{\C}^{\times}
\end{equation*}
where $\frak{S}_2\ltimes{\C}^{\times}$ is as defined in Lemma \ref{almost free g5} (3). 
It is easy to see the following $\frak{S}_2\ltimes{\C}^{\times}$-decomposition of $W$: 
\begin{equation*}
V_{2,0} = \langle X_0X_1\rangle \oplus \langle X_0^2,  X_1^2 \rangle,  
\end{equation*}
\begin{eqnarray*}
V_{3,1} &=& \langle X_0^3Y_1+X_1^3Y_0\rangle \oplus \langle X_0^3Y_1-X_1^3Y_0 \rangle \oplus  \\ 
             & & \langle X_0^3Y_0, \; X_1^3Y_1\rangle \oplus \langle X_0^2X_1Y_0, \; -X_0X_1^2Y_1 \rangle \oplus 
              \langle X_0X_1^2Y_0, \; X_0^2X_1Y_1\rangle. 
\end{eqnarray*}
Let $W_i$ be the representation of $\frak{S}_2\ltimes{\C}^{\times}$ induced from 
the weight $i$ scalar representation of ${\C}^{\times}$, 
$V_{\sigma}$ the sign representation of $\frak{S}_2$ pulled back to $\frak{S}_2\ltimes{\C}^{\times}$, 
and $V_0$ the trivial representation. 
By the above calculation we have the decomposition 
\begin{equation*}
W/{\C}v \simeq V_0 \oplus V_{\sigma} \oplus W_1^{\oplus2} \oplus W_2 \oplus W_3. 
\end{equation*}
Here notice that our ${\C}^{\times}$ is defined rather as the quotient by $-1$ of 
the $\alpha$-torus ${\C}^{\times}$ in \eqref{eqn:torus action g=17}, 
and this division by $-1$ reduces the weights of ${\C}^{\times}$-representations to half.  
Now $\frak{S}_2\ltimes{\C}^{\times}$ acts on ${\proj}(W_1\oplus W_2)$ almost freely 
so that we can apply the no-name lemma to the projection ${\proj}(W/{\C}v)\dashrightarrow{\proj}(W_1\oplus W_2)$. 
This gives  
\begin{equation*}
{\proj}(W/{\C}v)/\frak{S}_2\ltimes{\C}^{\times} \sim {\C}^6\times({\proj}(W_1\oplus W_2)/\frak{S}_2\ltimes{\C}^{\times}). 
\end{equation*}
Since ${\proj}(W_1\oplus W_2)/\frak{S}_2\ltimes{\C}^{\times}$ has dimension $2$, it is rational. 
Therefore $\mathbb{G}(1, {\proj}W)/{\aut}(X_{1,1})$ is stably rational of level $1$, 
and Proposition \ref{rational g5} is proved for $b=2$. 
\end{proof}



\begin{thebibliography}{99}

\bibitem{A-C}Arbarello, E.; Cornalba, M.
\textit{Footnotes to a paper of Beniamino Segre.} 
Math. Ann. \textbf{256} (1981), no. 3, 341--362. 

 
\bibitem{Bo}B\"ohning, C. 
\textit{The rationality problem in invariant theory.} 
arXiv:0904.0899, to appear from Progress in Math. 
   

\bibitem{B-B-C}B\"ohning, C.; Graf von Bothmer, H.-C.; Casnati, G.
\textit{Birational properties of some moduli spaces related to tetragonal curves of genus $7$.}
Int. Math. Res. Notices 2012, no. 22, 5219--5245.

\bibitem{B-K}Bogomolov, F. A.; Katsylo, P. I.
\textit{Rationality of some quotient varieties.} 
Mat. Sb. (N.S.) \textbf{126}(\textbf{168}) (1985), 584--589.

\bibitem{C-dC}Casnati, G.;Del Centina, A. 
\textit{The rationality of Weierstrass space of type $(4, g)$.} 
Math. Proc. Camb. Phil. Soc. \textbf{136} (2004), 53--66. 

\bibitem{dC-G}Del Centina, A.; Gimigliano, A.
\textit{Scrollar invariants and resolutions of certain $d$-gonal curves.} 
Ann. Univ. Ferrara Sez. VII (N.S.) \textbf{39} (1993), 187--201. 

\bibitem{Ge}Geiss, F. 
\textit{The unirationality of Hurwitz spaces of $6$-gonal curves of small genus.} 
Doc. Math. \textbf{17} (2012) 627--640. 

\bibitem{Ka1}Katsylo, P. I.
\textit{Rationality of the moduli spaces of hyperelliptic curves.} 
Izv. Akad. Nauk SSSR. \textbf{48} (1984), 705--710. 


\bibitem{Ka2}Katsylo, P. I. 
\textit{Rationality of fields of invariants of reducible representations of ${\SL}_2$.} 
Mosc. Univ. Math. Bull. \textbf{39} (1984) 80--83. 

\bibitem{Ko}Koitabashi, M. 
\textit{Automorphism groups of generic rational surfaces.} 
J. Algebra \textbf{116} (1988), no. 1, 130--142. 

\bibitem{Ma1}Ma, S. 
\textit{The rationality of the moduli spaces of trigonal curves of odd genus.} 
J. Reine. Angew. Math. 683 (2013), 181--187.


\bibitem{Ma2}Ma, S. 
\textit{Rationality of fields of invariants for some representations of ${\SLSL}$.} 
Compositio Math. 149 (2013) no.7, 1225--1234. 


\bibitem{Ma3}Ma, S. 
\textit{The rationality of the moduli spaces of trigonal curves.} 
arXiv:1207.0184. 


\bibitem{Mi}Miyata, T. 
\textit{Invariants of certain groups. I.} 
Nagoya Math. J. \textbf{41} (1971), 69--73. 


\bibitem{Pe}Petri, K. 
\textit{\"Uber die invariante Darstellung algebraischer Funktionen einer Ver\"anderlichen.} 
Math. Ann. \textbf{88} (1923), no. 3-4, 242--289. 


\bibitem{Sc}Schreyer, F.-O. 
\textit{Syzygies of canonical curves and special linear series.} 
Math. Ann. \textbf{275} (1986), 105--137. 


\bibitem{Se}Serre, J.-P. 
\textit{Linear representations of finite groups.} 
GTM \textbf{42}. Springer-Verlag, 1977.  

\bibitem{SB}Shepherd-Barron, N. I.
\textit{The rationality of certain spaces associated to trigonal curves.} 
Algebraic geometry, Bowdoin, 1985, 165--171, 
Proc. Symp. Pure Math., \textbf{46}, Part 1, Amer. Math. Soc., Providence, 1987.


\end{thebibliography}
\end{document}